\documentclass[10pt]{amsart}
\usepackage{graphicx}
\usepackage[mathscr]{eucal}
\usepackage{microtype}
\usepackage{amssymb}
\usepackage[parfill]{parskip}
\usepackage{mathrsfs}
\usepackage{xcolor}
\usepackage{color}
\usepackage[top=1in, bottom=1in, left=1in, right=1in]{geometry}
\usepackage{enumerate}
\usepackage{bm}
\usepackage[normalem]{ulem}
\usepackage{cancel}
\usepackage[all]{xy}
\usepackage{tikz}

\usepackage{setspace} 

\newcommand{\lr}{\mathnormal{\mathbf{L}(\mathbb R)}}
\newcommand{\gch}{\mathnormal{\mathsf{GCH}}}

\newcommand{\ulim}[2]{#1\text{-}\lim\limits_{#2<\omega}}
\newcommand{\usum}[2]{#1\text{-}\sum_{#2<\omega}}
\newcommand{\uadd}[2]{#1\text{-}\bigoplus_{#2<\omega}}

\newcommand{\zf}{\mathnormal{\mathsf{ZF}}}
\newcommand{\zfc}{\mathnormal{\mathsf{ZFC}}}
\newcommand{\zfa}{\mathnormal{\mathsf{ZFA}}}

\DeclareMathOperator{\fs}{FS}

\DeclareMathOperator{\dom}{dom}

\DeclareMathOperator{\fix}{Fix}
\DeclareMathOperator{\sym}{Sym}

\DeclareMathOperator{\aut}{Aut}

\newcommand{\leqrk}[2]{#1\leq_{\mathrm{RK}}#2}
\newcommand{\eqrk}[2]{#1\equiv_{\mathrm{RK}}#2}
\newcommand{\ch}{\mathnormal{\mathsf{CH}}}
\newcommand{\utr}{\mathnormal{\mathsf{UT}(\mathbb R)}}

\newcommand{\sub}{\subseteq}

\newcommand{\cF}{\mathscr{F}}
\newcommand{\ord}{\mathbf{Ord}}

\newtheorem{theorem}{Theorem}[section]
\newtheorem{proposition}[theorem]{Proposition}
\newtheorem{lemma}[theorem]{Lemma}
\newtheorem{corollary}[theorem]{Corollary}
\newtheorem{conjecture}[theorem]{Conjecture}
\newtheorem{question}[theorem]{Question}

\theoremstyle{definition}
\newtheorem{definition}[theorem]{Definition}

\theoremstyle{remark}

\author[D. Fern\'andez]{David Fern\'andez-Bret\'on}
\address{
Escuela Superior de F\'{\i}sica y Matem\'aticas\\
Instituto Polit\'ecnico Nacional\\
Av. Instituto Polit\'ecnico Nacional s/n Edificio 9, 
Col. San Pedro Zacatenco, Alcald\'{\i}a Gustavo A. Madero, 07738, CDMX, Mexico. 
}
\email{dfernandezb@ipn.mx}
\urladdr{https://dfernandezb.web.app}

\author[J. Navarro]{Jareb Navarro-Castillo}
\address{
Centro de Ciencias Matemáticas\\
Universidad Nacional Autónoma de México\\
Antigua Carretera a P\'atzcuaro 8701,
Col. Ex Hacienda San Jos\'e de la Huerta,
Morelia 58089, Michoac\'an, Mexico
}
\email{jareb@matmor.unam.mx}

\author[J. Soria]{Jes\'us A. Soria-Rojas}
\address{
Escuela Superior de F\'{\i}sica y Matem\'aticas\\
Instituto Polit\'ecnico Nacional\\
Av. Instituto Polit\'ecnico Nacional s/n Edificio 9, 
Col. San Pedro Zacatenco, Alcald\'{\i}a Gustavo A. Madero, 07738, CDMX, Mexico. 
}
\curraddr{
Department of Algebra\\
Faculty of Mathematics and Physics\\
Charles University\\
Sokolovská 83, 186 75 Praha 8, Czech Republic.
}
\email{rojas@math.cas.cz}

\title[Q-points, selectives, and idempotents]{Q-points, selective ultrafilters, and idempotents, \\ with an application to choiceless set theory}



\begin{document}

\maketitle

\begin{abstract}
We study ultrafilters from the perspective of the algebra in the \v{C}ech--Stone compactification of the natural numbers, and idempotent elements therein. The first two results that we prove establish that, if $p$ is a Q-point (resp. a selective ultrafilter) and $\mathscr F^p$ (resp. $\mathscr G^p$) is the smallest family containing $p$ and closed under iterated sums (resp. closed under Blass--Frol\'{\i}k sums and Rudin--Keisler images), then $\mathscr F^p$ (resp. $\mathscr G^p$) contains no idempotent elements. The second of these results about a selective ultrafilter has the following interesting consequence: assuming a conjecture of Blass, in models of the form $\lr[p]$ where $\lr$ is a Solovay model (of $\zf$ without choice) and $p$ is a selective ultrafilter, there are no idempotent elements. In particular, the theory $\zf$ plus the existence of a nonprincipal ultrafilter on $\omega$ does not imply the existence of idempotent ultrafilters, which answers a question of Di Nasso and Tachtsis (Proc. Amer. Math. Soc. 146, 397--411) that was also asked by Tachtsis (J. Symb. Log. 83, 557--571). Following the line of obtaining independence results in $\zf$, we finish the paper by proving that $\zf$ plus ``every additive filter can be extended to an idempotent ultrafilter'' does not imply the Ultrafilter Theorem over $\mathbb R$, answering another question of Di Nasso and Tachtsis from the aforementioned paper.
\end{abstract}

\section{Introduction}

The main contribution of this paper is providing one full answer, and a partial answer, respectively, to two questions of Di Nasso and Tachtsis from~\cite{dinasso-tachtsis} (the second question was also asked by Tachtsis in~\cite{tachtsis-ellis}). In the process of obtaining one of these answers, we prove some $\zfc$ results about the relation (more precisely, the lack thereof) between idempotent ultrafilters and Q-points or selective ultrafilters, when one also considers certain ultrafilter-building operations. These are of independent interest, relevant to the study of the algebra in the \v{C}ech--Stone compactification in the usual $\zfc$ context.

Recall that an ultrafilter on a set $X$ is a family of sets $u\subseteq\wp(X)$ that is closed under supersets and under finite intersections, and with the property that for every $Y\subseteq X$ either $Y\in u$ or $X\setminus Y\in u$ (and only one of these two options holds). One topologizes the set $\beta X$ of all ultrafilters on $X$ by declaring open every set of the form $\overline{A}=\{u\in\beta X\big|A\in u\}$ for $A\subseteq X$; this turns $\beta X$ into a compact Hausdorff space (with many other interesting topological properties). Identifying each $x\in X$ with the principal ultrafilter centred at $x$, $\{A\subseteq X\big|x\in A\}$, one obtains a dense copy of $X$ (equipped with the discrete topology) within $\beta X$. Furthermore, if we have a semigroup operation, denoted $+$, on $X$, that operation can be extended to all of $\beta X$ by the formula $u+v=\{A\subseteq X\big|\{x\in X\big|\{y\in X\big|x+y\in A\}\in v\}\in u\}$. Typically, this extended operation is not commutative (even if the original operation on $X$ was) but it is associative and, furthermore, for each fixed $v\in\beta X$ the mapping $u\longmapsto u+v$ is continuous; this particular combination of algebraic and topological properties is usually referred to by saying that $\beta X$ is a right-topological semigroup.

Whenever one has two sets $X,Y$, a function $f:X\longrightarrow Y$, and an ultrafilter $u\in\beta X$, the set $\{B\subseteq Y\big|f^{-1}[B]\in u\}$ is an ultrafilter on $Y$, known as the {\it Rudin--Keisler image} of $Y$ under $f$ and denoted by $f(u)$. The ultrafilter $f(u)$ is generated by sets of the form $f[A]$, where $A\in u$. The {\bf Rudin--Keisler order} of ultrafilters, denoted $\leqrk{}{}$, is the relation generated by specifying that $\leqrk{f(u)}{u}$ for every ultrafilter $u$ over a set $X$ and for every function $f$ with domain $X$. This relation happens to be a preorder among all ultrafilters; the equivalence relation induced by this preorder is the {\bf Rudin--Keisler equivalence}, denoted $\eqrk{}{}$. Thus $\eqrk{u}{v}$ if and only if $\leqrk{u}{v}$ and $\leqrk{v}{u}$; it is a classical theorem of M. E. Rudin and Solovay (independently) that $\eqrk{u}{v}$ if and only if there are $A\in u$, $B\in v$, and a bijective function $f:A\longrightarrow B$ such that (extending $f$ arbitrarily to all of $X$ we have) $f(u)=v$.

There are three main types of special ultrafilters that we will work with in this paper. If $X$ is a countable set, a {\bf Q-point} is a nonprincipal ultrafilter $u\in\beta X$ with the property that, for every partition of $X$, $X=\bigcup_{n<\omega}F_n$ into finite sets, there exists an $A\in u$ such that $(\forall n<\omega)(|A\cap F_n|\leq 1)$. We will say, on the other hand, that the nonprincipal ultrafilter $u\in\beta X$ is {\bf selective} if, for every partition of $X$ (into arbitrary pieces), $X=\bigcup_{n<\omega}X_n$, either $X_n\in u$ for some $n$, or there exists an $A\in u$ such that $(\forall n<\omega)(|A\cap X_n|\leq 1)$. An ultrafilter $u$ on a countable set $X$ is selective if and only if for every function $f$ with domain $X$, there exists an $A\in u$ such that $f\upharpoonright A$ is either constant or one-to-one; this happens if and only if $u$ is simultaneously a Q-point and a P-point\footnote{A P-point is an ultrafilter $u$ on a set $X$ with the property that for every function $f$ with domain $X$, there exists an $A\in u$ such that $f\upharpoonright A$ is either constant or finite-to-one. P-points will only be dealt with in this paper insofar as selective ultrafilters are P-points.}. Another extremely useful equivalent definition of selective ultrafilter involves the so-called {\bf Ramsey property}: an ultrafilter $p$ on $X$ is selective if and only if for every colouring $c:[X]^2\longrightarrow 2$ there exists an $A\in p$ that is $c$-homogeneous. The third special type of ultrafilter that we will consider is an idempotent: if $X$ comes equipped with a semigroup operation $+$ (in this paper, this essentially means if $X=\mathbb N$ or $X=\omega$ with the usual addition operation), then a nonprincipal ultrafilter $u\in\beta X$ is an {\bf idempotent} if $u+u=u$ (note that the only principal ultrafilter satisfying this property would be the one corresponding to the identity element $0$; by convention we only allow the word ``idempotent'' to refer to nonprincipal ultrafilters). All of the definitions and equivalences mentioned in this paragraph are so well-known as to be better labeled folklore.

In standard $\zfc$ set theory, it is known that the existence of Q-points is independent of the $\zfc$ axioms, and so is the existence of selective ultrafilters. On the other hand, the existence of (nonprincipal) idempotent ultrafilters on $\mathbb N$ is provable in $\zfc$ (this is known as the Ellis--Numakura lemma); the usual proof of this fact utilizes Zorn's Lemma in an essential way. Di Nasso and Tachtsis~\cite{dinasso-tachtsis} analyzed the provability of the existence of nonprincipal idempotent ultrafilters in $\zf$ (and so did Tachtsis~\cite{tachtsis-ellis} from a somewhat more general perspective). One of their main results is that, assuming the Ultrafilter Theorem for $\mathbb R$, denoted by $\utr$, there are nonprincipal idempotent elements. To do this, they define the notion of an additive filter. A filter $\mathscr F$ on $X$ (that is, a family $\mathscr F\subseteq\wp(X)$ closed under supersets and under finite intersections) is {\bf additive} if, for every ultrafilter $v\supseteq\mathscr F$, $\mathscr F\subseteq\mathscr F+v$ where the operation $+$ between two filters is defined exactly as it was for two ultrafilters. Di Nasso and Tachtsis's argument, using in an essential way the notion of additive filter, actually proves more than the mere existence of idempotent ultrafilters: they prove~\cite[Theorem 3.6]{dinasso-tachtsis} that $\utr$ implies that every additive filter can be extended to an idempotent ultrafilter. At the end of the paper, they leave a number of open questions, one of which is whether this implication is reversible~\cite[Question (2), p. 410]{dinasso-tachtsis}. Another such question, of fundamental importance to this paper, is~\cite[Question (5), p. 410]{dinasso-tachtsis} whether the existence of a nonprincipal ultrafilter on $\omega$ already implies the existence of an idempotent ultrafilter; this question was also asked by Tachtsis~\cite[Question 4]{tachtsis-ellis}. Note that, by~\cite[Theorem 2.1 (ii)]{tachtsis-ellis}, it is not possible to obtain a negative answer to that question by means of a Fraenkel--Mostowski permutation model of $\zfa$, so symmetric submodels of forcing extensions must be used; Tachtsis himself~\cite[pp. 565--568]{tachtsis-ellis} considered the model obtained by forcing with $[\omega]^\omega/\mathrm{Fin}$ over Feferman's model $\mathcal M$ (in which there are no nonprincipal ultrafilters on $\omega$) and showed that the corresponding generic filter is a nonprincipal nonidempotent ultrafilter on $\omega$ in M[G], but it is still open whether in this particular model there are any idempotent ultrafilters (see~\cite[Problem 5]{tachtsis-ellis}; in this paper we consider a similar model, in the sense that it is also obtained by forcing with $[\omega]^\omega/\mathrm{Fin}$ over a certain specific model of $\zf$).

The authors of this paper were originally motivated by answering the two questions of Di Nasso and Tachtsis's mentioned in the previous paragraph. We, however, ended up along the way delving deep into some $\zfc$ results about Q-points, selective ultrafilters, and idempotents that are of independent interest; the paper deals with these results first. In Section 2, we prove that, starting with a Q-point and building new ultrafilters by means of iterating the addition operation (appropriately defined to deal with transfinite iterations), one never builds an idempotent ultrafilter. In Section 3 we prove a similar-looking result whose proof is, nevertheless, significantly more complex: starting with a selective ultrafilter and building new ultrafilters by any combination of taking Blass--Frol\'{\i}k sums and Rudin--Keisler images, one never gets to build an idempotent ultrafilter. As a corollary of this we derive that, if a certain conjecture of Blass (to be stated in that section) holds, then in $\zf$ plus the existence of nonprincipal ultrafilters one cannot prove the existence of idempotent ultrafilters, so this constitutes a partial solution to the second question of Di Nasso and Tachtsis (partial because it depends on Blass's conjecture being true). Finally, in Section 4 we move completely to the $\zf$ terrain, by using standard forcing and symmetric model constructions to prove the consistency with $\zf$ of the statement that every additive filter can be extended to an idempotent ultrafilter yet $\utr$ fails (hence the statement that every additive filter can be extended to an additive ultrafilter does not imply $\utr$ under $\zf$); of course this answers the first question of Di Nasso and Tachtsis's.

The paper is organized so as to be useful to the largest possible audience: the reader who is exclusively interested in the standard study of the algebra in the \v{C}ech--Stone compactification under $\zfc$ can read only Sections 2 and 3; on the other hand, the reader exclusively interested in choiceless set theory and independence proofs in $\zf$ can read Sections 3 and 4 only (and, it goes without saying, anyone without strong opposition to any of these two topics can gleefully read all of the paper).

\section*{Acknowledgements}

The first and third authors were partially supported by internal grant SIP-20230355; additionally, the first author was supported by internal grant SIP-20240886 (both grants are from Instituto Polit\'ecnico Nacional). The second and third authors contributed to this work as part of their master's studies, supported by Conahcyt scholarships numbers 829938 and 820186, respectively.

\section{No idempotents from a Q-point}

In this section we prove that a Q-point on $\omega$ generates a free semigroup on $\beta\omega$, in the sense that, if $p$ is the Q-point, then the mapping $n\longmapsto\underbrace{p+\cdots+p}_{n\text{ times}}$ is injective. We in fact show that this holds not only of the original Q-point, but also of any other ultrafilter obtainable from that Q-point by means of iterated sums (even transfinitely). An immediate corollary will be that all such ultrafilters are not idempotents.

\subsection{Generating a free semigroup}

Recall that, given an $A\subseteq\omega$, its {\bf set of finite sums} is defined as
\begin{equation*}
    \fs(A)=\left\{\sum_{a\in F}a\bigg|F\subseteq A\text{ is finite nonempty}\right\}.
\end{equation*}
The following notions will be crucial in studying sets of finite sums $\fs(A)$ generated by certain sets $A$.

\begin{definition}
Let $A\subseteq\omega$.
\begin{enumerate}
    \item For $n\in\mathbb N$, we define $\fs_n(A)=\{a_1+\cdots+a_n\big|a_1,\ldots,a_n\in A\text{ are distinct}\}$.
    \item We say that $A$ {\bf has uniqueness of sums} if, whenever $a_1<\cdots<a_n$ and $b_1<\cdots<b_m$ are all elements of $A$, the equality $a_1+\cdots+a_n=b_1+\cdots+b_m$ implies that $n=m$ and $a_i=b_i$ for all $i$.
\end{enumerate}
\end{definition}

Hence we have $\fs(A)=\bigcup_{n\in\mathbb N}\fs_n(A)$, and $A$ has uniqueness of sums precisely when every element of $\fs(A)$ can be represented as a sum of elements of $A$ in a unique way (up to permutation).

\begin{definition}
Let $u\in\beta\omega$. Then,
\begin{enumerate}
    \item Given an $n\in\mathbb N$, we use the notation
\begin{equation*}
    u^n=\underbrace{u+\cdots+u}_{n\text{ times}}.
\end{equation*}
    \item We say that $u$ {\bf generates a free subsemigroup} if, for all $n,m\in\mathbb N$, we have $u^n=u^m$ if and only if $n=m$.
\end{enumerate}
\end{definition}

We use the notation $u^n$ instead of the other, at first sight more obvious, choice of $nu$ because the latter could easily be confused with the ultrafilter generated by the sets of the form $nA=\{na\big|a\in A\}$ with $A\in u$ (the Rudin--Keisler image of $u$ under the mapping $k\longmapsto nk$), which is not the intended object of study. The rest of this subsection establishes a very useful sufficient condition for an ultrafilter to generate a free subsemigroup.

\begin{lemma}\label{lem:fsnainun}
Let $u\in\beta\omega$ be a nonprincipal ultrafilter, let $n\in\mathbb N$, and let $A\in u$. Then $\fs_n(A)\in u^n$.
\end{lemma}

\begin{proof}
In order to prove that $\fs_n(A)$ is an element of $u^n$ we will show that $\fs_n(A)$ is positive with respect to $u^n$, that is, for every $B\in u^n$ we have $\fs_n(A)\cap B\neq\varnothing$ (in particular, it is not possible to have $\omega\setminus\fs_n(A)\in u^n$, which immediately yields the desired conclusion). This statement will be proved by induction on $n\in\mathbb N$; the result is obvious if $n=1$. Suppose now that the result is established for $n$ and let $B\in u^{n+1}=u^n+u$; by the definition of ultrafilter sum this means that
    \begin{equation*}
        C=\{x\in\omega\big|\{y\in\omega\big|x+y\in B\}\in u\}\in u^n.
    \end{equation*}
By induction hypothesis we have $\fs_n(A)\cap C\neq\varnothing$, so there is an element in that intersection; namely, there are $a_1<\cdots<a_n$ elements of $A$ such that
    \begin{equation*}
        D=\{y\in\omega\big|a_n<y\text{ and }a_1+\cdots+a_n+y\in B\}\in u.
    \end{equation*}
Therefore, the set $D\cap A$ belongs to $u$ and in particular is nonempty, meaning that we may pick an $a_{n+1}\in A$ such that $a=a_1+\cdots+a_n+a_{n+1}\in B$. In other words, $a\in\fs_{n+1}(A)\cap B$.
\end{proof}

We can now immediately deduce the following result.

\begin{corollary}
    Let $u\in\beta\omega$ be such that there is an $A\in u$ with uniqueness of sums. Then $u$ generates a free subsemigroup.
\end{corollary}

\begin{proof}
    For $n,m\in\mathbb N$ we have $\fs_n(A)\in u^n$ and $\fs_m(A)\in u^m$ by Lemma~\ref{lem:fsnainun}. Since $A$ has uniqueness of sums,
    $\fs_n(A)\cap\fs_m(A)=\varnothing$ unless $n=m$, therefore, $u^n\neq u^m$ unless $n=m$.
\end{proof}

\subsection{Adequate sets}

We now develop a slightly more involved way of showing that certain ultrafilters generate free subsemigroups. Recall that, given a sequence $\langle x_n\big|n<\omega\rangle$, a {\em sum subsystem} of the sequence is another (finite or infinite) sequence $\langle y_n\big|n\in I\rangle$ (with $I\in\omega+1$) such that there is a block sequence of finite sets $\langle H_n\big|n\in I\rangle$ (where ``block sequence'' means that $\max(H_n)<\min(H_{n+1})$ for all $n$) satisfying $y_n=\sum_{i\in H_n}x_i$ for each $n$.

\begin{definition}\label{def:goodset}
     Let $X\subseteq\omega$ be an infinite set, and let $A\subseteq\omega$. Let $\langle x_n\big|n<\omega\rangle$ be the increasing enumeration of the elements of $X$.
     \begin{enumerate}
         \item If $A\subseteq\fs(X)$, and $n\in\mathbb N$, we will denote the set
         \begin{equation*}
             \fs_n^X(A)=\left\{a_1+\cdots+a_n\in\fs_n(A)\big|\langle a_1,\ldots,a_n\rangle\text{ is a sum subsystem of }\langle x_n\big|n<\omega\rangle\right\},
         \end{equation*}
        and we denote $\fs^X(A)=\bigcup_{n\in\mathbb N}\fs_n^X(A)$.
        \item We say that $A$ is {\bf $X$-adequate} if $A\subseteq\fs(X)$ and, for every subsequence $\langle x_{n_k}\big| k<\omega\rangle$ of $\langle x_n\big|n<\omega\rangle$, there exists a unique $m\in\omega$ such that $x_{n_1}+\cdots+x_{n_m}\in A$.
     \end{enumerate}
\end{definition}

The main intuition behind the first part of Definition~\ref{def:goodset} is that, whenever $A\subseteq \fs(X)$, typically $A$ will not have the uniqueness of sums. For example, if $A=\{x_0+x_n\big|n\geq 1\}\cup\{x_n\big|n\geq 1\}$, then (note that $A$ is $X$-adequate, although) $x_0+x_1+x_2$ belongs to $\fs(A)$ but it is not uniquely represented as such. Restricting ourselves to $\fs^X(A)$ solves this issue, as seen in the next Lemma.

\begin{lemma}\label{lem:uniquenessfsnx}
    If $X$ has the uniqueness of sums and $A$ is $X$-adequate, then each element of $\fs^X(A)$ is uniquely represented as such. That is, whenever $\langle a_1,\ldots,a_n\rangle$ and $\langle b_1,\ldots,b_m\rangle$ are two sum subsystems of $\langle x_n\big|n<\omega\rangle$, then $a_1+\cdots+a_n=b_1+\cdots+b_m$ implies that $n=m$ and $a_i=b_i$ for all $i$.
\end{lemma}

\begin{proof}
    The hypotheses imply that there are $i_1<\ldots<i_{k_1}<i_{k_1+1}<\ldots<i_{k_2}<\ldots<i_{k_n}$ and $j_1<\ldots<j_{l_1}<j_{l_1+1}<\ldots<j_{l_2}<\ldots<j_{l_m}$ such that for each $t$, $a_t=x_{i_{k_{t-1}+1}}+\cdots+x_{i_{k_t}}$ and $b_t=x_{j_{l_{t-1}+1}}+\cdots+x_{j_{l_t}}$ (with the convention that $k_0=l_0=0$). From $\sum_{i=1}^n a_i=\sum_{j=1}^m b_j$ we obtain $\sum_{t=1}^{k_n} x_{i_t}=\sum_{t=1}^{l_m}x_{j_t}$, and from here (since $X$ has uniqueness of sums) we can conclude that $k_n=l_m$ and each $i_t=j_t$. Now consider any subsequence of $\langle x_n\big|n<\omega\rangle$ containing the sequence $\langle x_{i_t}\big|t\leq k_1\rangle$ as an initial segment. Since $A$ is $X$-adequate and $a_1,b_1\in A$, we conclude $k_1=l_1$; therefore $a_1=b_1$. Now considering any subsequence of $\langle x_n\big|n<\omega\rangle$ containing $\langle x_{i_t}\big|k_1<t\leq k_2\rangle$ as an initial segment, and using that $A$ is $X$-adequate and $a_2,b_2\in A$, we conclude that $k_2=l_2$ and therefore $a_2=b_2$. Continuing with this process, we may conclude that $m=n$ and $k_t=l_t$ for all $t$, and consequently $a_t=b_t$ for each $t$. 
\end{proof}

\begin{definition}
Given an infinite set $X\subseteq\omega$ and an ultrafilter $u\in\beta\omega$, we say that $u$ is {\bf $X$-adequate} if there exists an $X$-adequate set $A$ with $A\in u$ and, for each cofinite subset $Y\subseteq X$, we have $\fs(Y)\in u$. In this case, we will say that the set $A$ {\bf witnesses} that $u$ is $X$-adequate.
\end{definition}

It is not hard to see that, if $A$ is $X$-adequate and $Y\subseteq X$ is infinite, then $A\cap\fs(Y)$ is $Y$-adequate. From this observation, the reader may conclude that an ultrafilter $u$ is $X$-adequate if and only if for each cofinite $Y\subseteq X$ there is an $A_Y\in u$ such that $A_Y$ is $Y$-adequate.

\begin{lemma}\label{lem:fsnx}
    Let $X\subseteq\omega$ be an infinite set with uniqueness of sums, let $u\in\beta\omega$ be an $X$-adequate ultrafilter, and let $A\in u$ witness this fact. Then for each $n\in\mathbb N$, $\fs_n^X(A)\in u^n$.
\end{lemma}

\begin{proof}
    Mimicking the proof of Lemma~\ref{lem:fsnainun}, we prove by induction on $n\in\mathbb N$ that for every $B\in u^n$ $\fs_n^X(A)\cap B\neq\varnothing$, the case $n=1$ being trivial. So, taking $B\in u^{n+1}=u^n+u$, we have by definition that
        \begin{equation*}
        C=\{x\in\omega\big|\{y\in\omega\big|x+y\in B\}\in u\}\in u^n.
    \end{equation*}
By induction hypothesis we can pick an $a\in\fs_n^X(A)\cap C$, that is, it is possible to write $a=a_1+\cdots+a_n$ such that $\langle a_1,\cdots,a_n\rangle$ is a sum subsystem of $\langle x_n\big|n<\omega\rangle$, and furthermore
    \begin{equation*}
        D=\{y\in\omega\big|a_n<y\text{ and }a_1+\cdots+a_n+y\in B\}\in u.
    \end{equation*}
The hypothesis implies that, if $Y$ is any cofinite subset of $X$, then $\fs(Y)\cap D\cap A$ belongs to $u$ and therefore it is nonempty, so we can pick an $a_{n+1}$ in that set. In particular, if $Y$ is the subset of $X$ containing only those elements larger than any of the ones appearing in the expressions of  $a_1,\ldots,a_n$ (which is well-defined since $X$ has uniqueness of sums), then $\langle a_1,\ldots,a_n,a_{n+1}\rangle$ will be a sum subsystem of $\langle x_n\big|n<\omega\rangle$; this way, $a_1+\cdots+a_n+a_{n+1}\in\fs_{n+1}^X(A)\cap B$.
\end{proof}

\begin{corollary}\label{prop:qpuntonoidem}
    Let $X\subseteq\omega$ be a set with uniqueness of sums, and let $u\in\beta\omega$ be an $X$-adequate ultrafilter. Then $u$ generates a free subsemigroup.
\end{corollary}

\begin{proof}
The hypotheses imply that, given two distinct $n,m\in\mathbb N$, by Lemma~\ref{lem:fsnx} we have $\fs_n^X(A)\in u^n$ and $\fs_m^X(A)\in u^m$; now Lemma~\ref{lem:uniquenessfsnx} guarantees that $\fs_n^X(A)$ is disjoint from $\fs_m^X(A)$ and therefore $u^n\neq u^m$.
\end{proof}

\subsection{Adding ultrafilters transfinitely}

It is easy to see that, whenever $u$ generates a free subsemigroup, so does every iterated addition of $u$ (i.e. every ultrafilter of the form $u^n$). The aim of this subsection is to show that, under certain circumstances, the same can be said of ``transfinite'' iterated additions of $u$. In order to make precise the notion of ``adding transfinitely'', we will use Rudin--Keisler images, as well as what is known as {\it Blass--Frol\'{\i}k sums}, developed by Frolík~\cite{frolik} and independently by Blass~\cite[Chapter IV]{blass-thesis}.

\begin{definition}
     Given $u\in\beta\omega$ and a sequence $\langle u_n\mid n<\omega\rangle$ of elements of $\beta\omega$, the {\bf Blass--Frolík sum} of the sequence of $u_n$ indexed by $u$ is the ultrafilter on $\omega\times\omega$ defined by
     \begin{equation*}
         \usum{u}{n}u_n=\left\{A\sub\omega\times\omega\big|\left\{n<\omega\big|\{m<\omega\big|(n,m)\in A\}\in u_n\right\}\in u\right\}.
     \end{equation*}
\end{definition}

Note that, in the particular case where the sequence of $u_n$ is constant, say with constant value $v$, then $\usum{u}{n}u_n$ is what is usually known as the {\it tensor product}, or the {\it Fubini product}, of $u$ and $v$. It is not hard to work out from the definition that sets of the form $\{(n,m)\in\omega\times\omega\big|n\in A\text{ and } m\in A_n\} $, where $A\in u$ and each $A_n\in u_n$, form a basis for the ultrafilter $\usum{u}{n}u_n$.

\begin{definition}
     \label{def:imagenrksuma}
   For an ultrafilter $u\in\beta\omega$ and a sequence $\langle u_n\big| n<\omega\rangle$ of ultrafilters on $\omega$, we define the {\bf sum} of the $u_n$ indexed by $u$, denoted $\uadd{u}{n}u_n$ to be the Rudin--Keisler image of $\usum{u}{n}{u_n}$ under the addition mapping (that is, the mapping $:\omega\times\omega\longrightarrow\omega$ given by $(n,m)\longmapsto n+m$).
\end{definition}

So, by definition, $\uadd{u}{n}u_n=\left\{A\subseteq\omega\big|\{(n,m)\in\omega\times\omega\mid n+m\in A\}\in\usum{u}{n}u_n\right\}$; one can then easily work out that $\uadd{u}{n}u_n$ admits a basis of sets of the form $\{n+m\big|n\in A\text{ and } m\in A_n\} $, where $A\in u$ and each $A_n\in u_n$. In particular, considering an ultrafilter $v$ and the sequence with constant value $v$ we obtain $\uadd{u}{n}{v}=u+v$, thus recovering the usual sum operation in $\beta\omega$. Now, the key fact about indexed sums is that they preserve adequateness, in the precise sense stated below.

\begin{theorem}\label{thm:preserve-goodness}
    Let $X\subseteq\omega$ be a set with uniqueness of sums, and let $u\in\beta\omega$ and $\langle u_n\big|n<\omega\rangle$ be a sequence of ultrafilters on $\omega$. If $u$ and each of the $u_n$ are $X$-adequate, then so is $\uadd{u}{n}u_n$.
\end{theorem}

\begin{proof}
    Let $A\in u$ witness that $u$ is $X$-adequate. For each $n\in A$, let $Y_n\subseteq X$ be the (cofinite) subset of $X$ containing only those elements larger than any involved in the expression of $n$ as a finite sum of $X$; and let $A_n\in u_n$ witness that $u_n$ is $Y_n$-adequate. We define $B=\{n+m\big|n\in A\text{ and }m\in A_n\}\in\uadd{u}{n}u_n$. It is clear that $B\subseteq\fs(X)$; furthermore, if $\langle x_{n_k}\big|k<\omega\rangle$ is any subsequence of $\langle x_n\big|n<\omega\rangle$, then (since $A$ is $X$-adequate) there exists a unique $k_0$ such that $a=x_{n_0}+\cdots+x_{n_{k_0}}\in A$. Then we know that $\{ x_n\big|n>n_{k_0}\}=Y_a$ and so, since $A_a$ is $Y_a$-adequate, there exists a unique $k$ such that $b=x_{n_{k_0}+1}+\cdots+x_{n_k}\in A_a$. This means that $k$ is the unique number such that $x_{n_0}+\cdots+x_{n_k}=a+b\in B$, so that $B$ is $X$-adequate.
    To finish off, note that the preceding argument works for any $X$ that has the uniqueness of sums and such that there is an $A\in u$ that is $X$-adequate. Therefore, by running the same argument with an arbitrary cofinite $Y\subseteq X$ we would be building the set $B$ in such a way that $B\subseteq\fs(Y)$. Hence, $\fs(Y)\in u$, and we are done. 
\end{proof}

We now state another property of adequate transfinite sums. Begin by noting that, by definition, $\leqrk{\uadd{u}{n}u_n}{\usum{u}{n}u_n}$ for any ultrafilters $u,u_n$. The following proposition improves this whenever ``adequate'' ultrafilters are added.

\begin{proposition}\label{prop:goodnessimpliesrkequiv}
    Let $X\subseteq\omega$ be a set with uniqueness of sums, and let $u\in\beta\omega$ and $\langle u_n\big|n<\omega\rangle$ be a sequence of ultrafilters on $\omega$. If $u$ and each of the $u_n$ are $X$-adequate, then
    \begin{equation*}
        \eqrk{\usum{u}{n}u_n}{\uadd{u}{n}u_n}
    \end{equation*}
\end{proposition}
\begin{proof}
    Since $\uadd{u}{n}u_n$ is by definition the Rudin--Keisler image of $\usum{u}{n}u_n$ under the addition mapping, it suffices to exhibit an element of the latter ultrafilter in which this mapping is injective by~\cite[Theorem 8.17]{hindman-strauss}. To do this, we mimic the proof of Theorem~\ref{thm:preserve-goodness}: Let $A\in u$ witness that $u$ is $X$-adequate; then for each $a\in A$ let $Y_a\subseteq X$ be the cofinite subset of elements larger than any that appears in the expression of $a$ as a finite sum from $X$, and let $A_a\in u_a$ witness that $u_a$ is $Y_a$-adequate. We define 
    \begin{equation*}
        B=\left\{(a,b)\big|a\in A\text{ and }b\in A_a\right\},
    \end{equation*}
    a basic element of $\usum{u}{n}u_n$. Let us show that addition restricted to $B$ is injective: suppose $(a,b),(c,d)\in B$ satisfy $a+b=c+d$. The definition of $B$ implies that there are $i_1<\cdots<i_{k_0}<i_{k_0+1}<\cdots<i_{k_1}$ and $j_1<\cdots<j_{l_0}<j_{l_0+1}<\cdots<j_{l_1}$ such that $a=x_{i_1}+\cdots+x_{i_{k_0}}$, $b=x_{i_{k_0 + 1}}+\cdots+x_{i_{k_1}}$, $c=x_{j_1}+\cdots+x_{j_{l_0}}$ and $d=x_{j_{l_0 +1}}+\cdots+x_{j_{l_1}}$. From $a+b=c+d$ we deduce that $k_1=l_1$. Furthermore, since $a,c\in A$ and $A$ is $X$-adequate, by considering any subsequence of $\langle x_n\big|n<\omega\rangle$ containing $\langle x_{i_1},\ldots,x_{i_{k_1}}\rangle$ as an initial segment we may conclude that $a=c$. Repeating the same reasoning with $b,d$ and the sequence $\langle x_{i_{k_0 + 1}},\ldots,x_{i_{k_1}}\rangle$ we may conclude that $b=d$, and we are done.
\end{proof}

Proposition~\ref{prop:goodnessimpliesrkequiv} shows that our choice of definition for a transfinite sum is extremely natural. With this, we can extend the idea of the kinds of ultrafilters that can be generated from a given one by means of addition.

\begin{definition}
     For an ultrafilter $p\in\beta\omega$, we define by transfinite recursion
    \begin{eqnarray*}
    \cF_0^p & = & \{p\}, \\
    \cF_{\alpha+1}^p & = & \left\{\uadd{u}{n}u_n\big|u\in\cF_\alpha^p\text{ and }u_n\in\cF_\alpha^p\text{ for all }n<\omega\right\}, \\
    \cF_\alpha^p & = & \bigcup_{\xi<\alpha}\cF_\xi^p,\ \ \ \ \text{ if }\alpha\text{ is limit}; \\
    \cF^p & = & \bigcup_{\alpha\in\ord}\cF_\alpha^p.
    \end{eqnarray*}
\end{definition}

Note that, since $\omega_1$ has uncountable cofinality, it is easy to prove that $\cF^p=\cF_{\omega_1}^p$. Our previous results have already laid out the tools needed to prove the following.

\begin{proposition}
\label{prop:sumassonqbuenas}
    Let $X$ be a set with uniqueness of sums, and let $p\in\beta\omega$ be an $X$-adequate ultrafilter. Then, every element of $\cF^p$ is $X$-adequate.
\end{proposition}
\begin{proof}
     The proof is by induction on $\alpha$ such that our ultrafilter belongs to $\cF_\alpha^p$, the cases $\alpha=0$ and $\alpha$ limit being trivial. The only nontrivial case, the successor step, is handled by Theorem~\ref{thm:preserve-goodness}.
\end{proof}

\begin{corollary}\label{cor:good-implies-noidem}
    Let $X\subseteq\omega$ be a set with uniqueness of sums, and let $p\in\beta\omega$ be an $X$-adequate ultrafilter. Then, every element of $\cF^p$ generates a free subsemigroup.
\end{corollary}

\begin{proof}
    Immediate from Corollary~\ref{prop:qpuntonoidem} and Proposition~\ref{prop:sumassonqbuenas}.
\end{proof}

This leads to an extremely interesting result about Q-points, arguably the main result of this section. It generalizes the well-known observation that Q-points cannot be idempotent elements of $\beta\omega$.

\begin{theorem}
\label{thm:nohayidem}
    Let $p\in\beta\omega$ be a Q-point. Then, every element of $\cF^p$ generates a free subsemigroup.
\end{theorem}

\begin{proof}
    Consider the partition of $\omega$ into finite sets $I_n$ given by $I_n=[2^n,2^{n+1})$. Since $p$ is a Q-point, there exists a $Y\in p$ such that, for each $n<\omega$, $Y\cap I_n$ contains at most one point. Now if $Y=\{y_n\big|n<\omega\}$ is its increasing enumeration, there is an $i<2$ such that $X=\{y_{2n+i}\big|n<\omega\}\in p$. Letting $x_n=y_{2n+i}$ for all $n$, we get that $X=\{x_n\big|n<\omega\}$ is the increasing enumeration of that set, and it satisfies $2x_n<x_{n+1}$. The latter property allows proving inductively that $x_n>\sum_{k<n}x_k$, which easily yields that $X$ has the uniqueness of sums. Since $X\in p$, for every cofinite $Y\subset X$, $Y\in p$ and then $\text{FS}(Y)\in p$. Also, $X$ witnesses that $p$ is $X$-adequate, because it has uniqueness of sums.  Therefore the result follows from Corollary~\ref{cor:good-implies-noidem}.
\end{proof}
   
\section{Free subsemigroups from a selective}

This section follows a similar line as the previous one, in the sense that we show how certain ultrafilters built from an initial one generate free semigroups; we utilize a wider definition of ``ultrafilters built from'' and, in exchange, have to assume stronger properties regarding the initial ultrafilter.

\subsection{$u$-limits}

Recall that, in a topological space, a point $x$ is the $u$-limit of the sequence of points $x_n$ (where $u$ is an ultrafilter over $\omega$) if for each neighbourhood $V$ of $x$ we have $\{n<\omega\big|x_n\in V\}\in u$; in a compact space $u$-limits always exist (and they are unique if the space is Hausdorff). Hence, in the context of the \v{Cech}--Stone compactification, $u$-limits always exist and are unique, and they can be defined combinatorially as follows.

\begin{definition}\label{def:ulim}
Let $u$ and $u_n$ ($n<\omega$) be ultrafilters on a set $X$. We denote by $\ulim{u}{n}u_n$ the unique ultrafilter $v$ with the property that, for $A\subseteq X$, we have
\begin{equation*}
    A\in v\iff\{n<\omega\big|A\in u_n\}\in u.
\end{equation*}
This ultrafilter will be called {\em the $u$-limit of the sequence of $u_n$}.
\end{definition}

It is a simple exercise to show that $\ulim{u}{n}u_n$ is indeed an ultrafilter. It is also straightforward to prove that $\ulim{u}{n}u_n$ is the collection of all sets of the form
\begin{equation*}
    \bigcup_{n\in A}A_n,
\end{equation*}
where $A\in u$ and $A_n\in u_n$ for all $n<\omega$.

There is a sense in which the notion of a $u$-limit generalizes the Blass--Frol\'{\i}k sum: if $u$ and $u_n$ ($n<\omega$) are ultrafilters over $\omega$, then $\usum{u}{n}u_n=\ulim{u}{n}v_n$, where $v_n$ is the Rudin--Keisler image of $u_n$ under the (injective) mapping $k\longmapsto (n,k)$ (note that $\eqrk{v_n}{u_n}$ for all $n$). Conversely, it is possible to define $u$-limits in terms of Blass--Frol\'{\i}k sums and Rudin--Keisler images: the reader might gladly check that $\ulim{u}{n}u_n$ is the Rudin--Keisler image of the ultrafilter $\usum{u}{n}u_n$ under the mapping $(x,y)\longmapsto y$.

Consequently, a family of ultrafilters that is closed under Rudin--Keisler images will be closed under Blass--Frol\'{\i}k sums if and only if it is closed under ultrafilter limits. The same applies to the ``transfinite sum'' defined in the previous section, since $\uadd{u}{n}u_n=\ulim{u}{n}(n+u_n)$, where $n+u_n$ denotes the Rudin--Keisler image of the ultrafilter $u_n$ under the translation mapping $k\longmapsto n+k$. In particular, sums of ultrafilters (in the sense of the algebra in the \v{C}ech--Stone compactification) are $u$-limits, v.gr. $u+v=\ulim{u}{n}(n+v)$. In particular, a family of ultrafilters closed under Rudin--Keisler images and under Blass--Frol\'{\i}k sums (equivalently replace the latter with ultrafilter limits) will be a subsemigroup of $\beta\omega$.

An important feature of $u$-limits that we will use frequently is that they only depend on the sequence ``up to $u$-many elements''. In other words, if $\langle u_n\big|n<\omega\rangle$ and $\langle v_n\big|n<\omega\rangle$ are two sequences such that $\{n<\omega\big|u_n=v_n\}\in u$, then $\ulim{u}{n}u_n=\ulim{u}{n}v_n$.

Just like regular limits, $u$-limits are also preserved by continuous functions. In particular, whenever we have a function $f:X\longrightarrow Y$ and ultrafilters $u_n$ over $X$ ($n\in\omega$), the equation
\begin{equation}
    f(\ulim{u}{n}u_n)=\ulim{u}{n}f(u_n)
\end{equation}
holds. Using this equation iteratively, one can obtain the following for nested $u$-limits.

\begin{lemma}\label{lem:iteratedulims}
    Given sequences of ultrafilters $\langle u_n\big|n<\omega\rangle$, $\langle v_n\big|n<\omega\rangle$, we have
\begin{equation*}
\ulim{\left(\ulim{p}{n}u_n\right)}{m} v_m=\ulim{p}{n}\left(\ulim{u_n}{m} v_m\right)
\end{equation*}
\end{lemma}

\begin{proof}
    Suppose $X\in\ulim{\left(\ulim{p}{n}u_n\right)}{m}v_m$. Then,
    \begin{equation*}
        A=\{m<\omega\big|X\in v_m\}\in\ulim{p}{n} u_n,
    \end{equation*}
    which means that $B=\{n<\omega\big|A\in u_n\}\in p$. So, if $n\in B$ then we have
    \begin{equation*}
        \{m<\omega\big|X\in v_m\}=A\in u_n,
    \end{equation*}
    which means $X\in\ulim{u_n}{m}v_m$. The conclusion is that
    \begin{equation*}
        \left\{n<\omega\big|X\in\ulim{u_n}{m}v_m\right\}\supseteq B\in p,
    \end{equation*}
    meaning that $X\in\ulim{p}{n}\left(\ulim{u_n}{m} v_m\right)$. Therefore we have
    \begin{equation*}
    \ulim{\left(\ulim{p}{n}u_n\right)}{m} v_m\subseteq\ulim{p}{n}\left(\ulim{u_n}{m} v_m\right);
    \end{equation*}
    by maximality of ultrafilters, this set inclusion is actually an equality.
\end{proof}

We finish this subsection with the following useful and easy consequence of Lemma~\ref{lem:iteratedulims}.

\begin{corollary}\label{lem:f(u)limit}
    Let $f:\omega\longrightarrow\omega$ be a function, and let $\langle u_n\big|n<\omega\rangle$ be a sequence of ultrafilters. Then,
    \begin{equation*}
        \ulim{f(u)}{n}u_n=\ulim{u}{n}u_{f(n)}
    \end{equation*}
\end{corollary}

\begin{proof}
    Within the proof we identify elements of $\omega$ with their corresponding principal ultrafilters. Looking at the definition of the Rudin--Keisler image (and using the simple fact that $u=\ulim{u}{n}n$), the reader should convince herself that $f(u)=\ulim{u}{m}f(m)$. Therefore
    \begin{equation*}
    \ulim{f(u)}{n}u_n=\ulim{\left(\ulim{u}{m}f(m)\right)}{n}u_n=\ulim{u}{m}\left(\ulim{f(m)}{n} u_n\right)=\ulim{u}{m}u_{f(m)},
    \end{equation*}
    by Lemma~\ref{lem:iteratedulims}.
\end{proof}

\subsection{A theorem about $p$-limits when $p$ is selective}

Normally, in a topological space different sequences might have the same limit; the same is true of $u$-limits. In this subsection we will see that, in the particular case of taking limits along a selective ultrafilter, this only happens in the trivial case (i.e. when one of the limits is of the constant sequence with constant value the other limit).

\begin{lemma}\label{lem:combinations}
Let $X$ be a set, let $\mathscr X$ be a finite family of subsets of $X$, let $u$ be an ultrafilter over $X$, and let $A\in u$. If we are given ultrafilters $v_1,\ldots,v_n$ over $X$, all of them distinct from $u$, then we can find an $A'\subseteq A$ such that:
\begin{enumerate}
\item $A'\in u$,
\item $A'\notin v_i$ for all $i$,
\item For every $B\in\mathscr X$, either $B\cap A'=\varnothing$, or $A'\subseteq B$.
\end{enumerate}
\end{lemma}

\begin{proof}
For each $B\in\mathscr X$ and each $i\in\{1,\ldots,n\}$ define a set $A_{B,i}$ as follows:  we first let $C_{B,i}\in\{B,X\setminus B\}$ be such that $C_{B,i}\in v_i$; if $C_{B,i}\notin u$ then we define $A_{B,i}=X\setminus C_{B,i}$ and otherwise (since $u\neq v_i$) we may pick $A_{B,i}\subseteq C_{B,i}$ such that $A_{B,i}\in u\setminus v_i$. It is readily checked that $\displaystyle{A'=A\cap\left(\bigcap_{B\in\mathscr X,i\in\{1,\ldots,n\}}A_{B,i}\right)}$ satisfies the required conditions.
\end{proof}

Since $u$-limits only depend on the relevant sequence up to $u$-many elements, we introduce the following definitions to capture this idea.

\begin{definition}\label{def:modulou}
Let $u$ be an ultrafilter. We will say that the sequence $\langle u_n\big|n<\omega\rangle$ is
\begin{enumerate}
    \item {\bf $u$-constant}, if there exists an ultrafilter $v$ such that $\{n<\omega\big|v=u_n\}\in u$; and
    \item {\bf $u$-injective} if there exists an $A\in u$ such that whenever $n,m\in A$ are distinct, $u_n\neq u_m$.
\end{enumerate}
We will also say that two sequences $\langle u_n\big|n<\omega\rangle$ and $\langle v_n\big|n<\omega\rangle$ are {\bf $u$-equal} if $\{n<\omega\big|u_n=v_n\}\in u$.
\end{definition}

So we always have $\ulim{u}{n}u_n=\ulim{u}{n}v_n$ whenever the sequence of the $u_n$ is $u$-equal to the sequence of the $v_n$.

In the particular case where $p$ is a selective ultrafilter, every sequence will be either $p$-constant or $p$-injective; the $p$-limits of two $p$-constant sequences are equal precisely when both sequences have the same (constant) value modulo $p$; the following theorem establishes the analogous fact for two $p$-injective sequences.

\begin{theorem}\label{thm:uniquenessofsequence}
Let $p$ be a selective ultrafilter, and let $\langle u_n\big|n<\omega\rangle,\langle v_n\big|n<\omega\rangle$ two $p$-injective sequences. Then, we have $\ulim{p}{n}u_n=\ulim{p}{n}v_n$ if and only if $\{n<\omega\big|u_n=v_n\}\in p$.
\end{theorem}

\begin{proof}
    In the nontrivial direction, we will reason by contrapositive and so start by assuming, without loss of generality, that $u_n\neq v_n$ for all $n$; of course the hypothesis that both sequences are $p$-injective means the $u_n$ are pairwise distinct, and similarly the $v_n$. Furthermore, since $p$ is a selective ultrafilter we may assume (by an application of the Ramsey property) that either for all $n<m$ we have $u_n=v_m$, or for all $n<m$ we have $u_n\neq v_m$; the former readily leads to a contradiction so the conclusion is that we may assume without loss of generality that whenever $n\neq m$ we have $u_n\neq v_m$. We will proceed to build an element $X\in p$ and two  sequences of sets $\langle A_n\big|n\in X\rangle$, $\langle B_n\big|n\in X\rangle$ such that:
\begin{enumerate}
 \item If $n,m\in X$ are distinct, then $A_n\cap A_m=\varnothing=B_n\cap B_m$,
 \item for $n,m\in X$ distinct, we have $A_n\in u_n\setminus v_n$ and $A_n\notin u_m\cup v_m$; and similarly $B_n\in v_n\setminus u_n$ and $B_n\notin u_m\cup v_m$,
 \item for each $n\in X$, $A_n\cap B_n=\varnothing$,
 \item for $n,m\in X$ such that $n<m$, either $A_n\cap B_m=\varnothing$ or $B_m\subseteq A_n$ and similarly either $A_m\cap B_n=\varnothing$ or $A_m\subseteq B_n$.
\end{enumerate}
Begin by picking, for each $n<\omega$, two disjoint sets $A_n''',B_n''$ such that $A_n'''\in u_n$ and $B_n''\in v_n$. Now, by recursion on $n<\omega$ get $A_n''\subseteq A_n'''$ such that $A_n''\in u_n$ and, for each $i<n$, $A_n''\notin u_i$ and either $A_n''\cap A_i''=\varnothing$, or $A_n''\subseteq A_i''$; this is simply an application of Lemma~\ref{lem:combinations} with $\mathscr X=\{A_i''\big|i<n\}$. This way, the sequence of sets $A_n''$ satisfies that, for $n<m<\omega$, $A_m''\in u_m\setminus u_n$, and either $A_m''\subseteq A_n''$ or $A_n''\cap A_m''=\varnothing$. Now define a colouring $c_1:[\omega]^2\longrightarrow 2$ by letting $c_1(\{n,m\})=1$ if and only if $A_n''\cap A_m''=\varnothing$; since $p$ is a selective ultrafilter there exists a $c_1$-homogeneous set $X'\in p$. If $X'$ is $c_1$-homogeneous in colour $1$ then let $A_n'=A_n''$ for each $n\in X'$; otherwise, for each $n\in X'$ let $m=\min(X'\setminus(n+1))$ and define $A_n'=A_n''\setminus A_m''$. Note that the sequence of sets $A_n'$ ($n\in X'$) thus created is pairwise disjoint and satisfies $A_n'\in u_n$ for every $n\in X'$. We now proceed to recursively choose, for $n\in X'$, sets $B_n'\subseteq B_n''$ such that $B_n'\in v_n$ and, for each $i\in X'\cap n$, $B_n'\notin v_i$, $B_n'\notin u_i$ and either $B_n'\cap B_i'=\varnothing$ or $B_n'\subseteq B_i'$ and also either $B_n'\cap A_i'=\varnothing$ or $B_n'\subseteq A_i'$; this is again an application of Lemma~\ref{lem:combinations} with $\mathscr X=\{A_i'\big|i\in X'\cap n\}\cup\{B_i'\big|i\in X'\cap n\}$. Define the colouring $c_2:[X']^2\longrightarrow 2$ by letting $c_2(\{n,m\})=1$ if and only if $B_n'\cap B_m'=\varnothing$; $p$ is selective so there is a $c_2$-homogeneous set $X\subseteq X'$ with $X\in p$. If $X$ is $c_1$-homogeneous in colour $1$ we let $B_n=B_n'$ for each $n\in X$; otherwise we define $B_n=B_n'\setminus B_m'$ where $m=\min(X\setminus(n+1))$. At this point, the sequences of sets $A_n',B_n$ ($n\in X$) are (each of them separately) pairwise disjoint and have the property that $A_n'\in u_n$, $B_n\in v_n$, $A_n'\cap B_n=\varnothing$, and if $n<m$ ($n,m\in X$) then either $B_m\subseteq A_n'$ or $B_m\cap A_n'=\varnothing$.

Finally, use again Lemma~\ref{lem:combinations} recursively on $X$ to obtain (for each $n\in X$) sets $A_n\subseteq A_n'$ with $A_n\in u_n$ and such that for $i<n$ either $A_n\subseteq B_i$ or $A_n\cap B_i=\varnothing$, and $A_n\not\in v_i$. The two sequences $\langle A_n\big|n\in X\rangle$, $\langle B_n\big|n\in X\rangle$ thus obtained satisfy the four requirements described above.

We now proceed to define yet another colouring $c:[X]^2\longrightarrow 2\times 2$ as follows: given $n<m$, for $n,m\in X$, we let $c(\{n,m\})=(i,j)$ where $i=1$ if and only if $A_n\cap B_m=\varnothing$ and $j=1$ if and only if $B_n\cap A_m=\varnothing$. Since $p$ is selective, there is a $c$-homogeneous set $Z\subseteq X$ with $Z\in p$; we now argue that $Z$ must be $c$-homogeneous in colour $(1,1)$. To see this, let $(i,j)$ the colour that $Z$ is homogeneous in, and take three distinct $n,m,k\in Z$ with $n<m<k$. If $i=0$ then we should have $B_k\subseteq A_n$ and also $B_k\subseteq A_m$, contradicting that $A_n\cap A_m=\varnothing$; similarly if $j=0$ then it should be the case that $A_k\subseteq B_n$ and $A_k\subseteq B_m$ which contradicts $B_n\cap B_m=\varnothing$. Hence $(i,j)=(1,1)$. So whenever $n,m\in Z$ it must be the case that $A_n\cap B_m=\varnothing=A_m\cap B_n$. In particular, the set $\bigcup_{n\in Z}A_n$, which belongs to $\ulim{p}{n}u_n$, is disjoint from the set  $\bigcup_{n\in Z}B_n$, which belongs to $\ulim{p}{n}v_n$; of course this implies $\ulim{p}{n}u_n\neq\ulim{p}{n}v_n$, and we are done.
\end{proof}

\subsection{A cancellative subsemigroup built from a selective}

We will now formalize the notion of starting from $p$ and iteratively (even transfinitely) performing Rudin--Keisler images and Blass--Frol\'{\i}k sums (equivalently, taking ultrafilter limits). The definition below is not as comprehensive at first sight, but we will eventually show (cf. Corollary~\ref{cor:characterizationofg}) that it formalizes this notion.

\begin{definition}
Given an ultrafilter $p\in\beta\omega$, we define families of ultrafilters on $\omega$, $\mathscr G_\alpha^p$, as follows: begin by letting $\mathscr G_1^p$ be the family of all principal ultrafilters on $\omega$, and recursively let
\begin{eqnarray*}
\mathscr G_{\alpha+1}^p & = & \left\{\ulim{p}{n}u_n\big|(\forall n<\omega)(u_n\in\mathscr G_\alpha^p)\right\}\ \\
\mathscr G_{\alpha}^p & = & \bigcup_{\xi<\alpha}\mathscr G_\xi^p\ \ \ \ \ \text{ if }\alpha\text{ is a limit ordinal.}
\end{eqnarray*}
Finally, we let $\mathscr G^p$ be the union of all of the $\mathscr G_\alpha^p$.
\end{definition}

So, for example, $\mathscr G_2^p$ will contain all $p$-limits of sequences of principal ultrafilters---that is to say, by identifying a sequence of principal ultrafilters with a function on $\omega$, $\mathscr G_2^p$ will contain all ultrafilters Rudin--Keisler below $p$. In the particular case where $p$ is a selective ultrafilter, $\mathscr G_2^p$ will contain only the ultrafilters that are either principal, or Rudin--Keisler equivalent to $p$. In any case, after that $\mathscr G_3^p$ contains $p$-limits of sequences of elements of $\mathscr G_2^p$, and so on. Similar to the case of the $\cF_\alpha^p$ from the previous section, it is easy to show that $\mathscr G_{\omega_1+1}^p=\mathscr G_{\omega_1}^p$ and, consequently, $\mathscr G^p=\mathscr G_{\omega_1}^p$. The following lemma records two basic properties of the sequence of $\mathscr G_\alpha^p$.

\begin{lemma}\label{lem:galphaproperties}
Let $p\in\beta\omega$ be an ultrafilter, and let $\alpha\leq\beta$ be ordinal numbers. Then,
    \begin{enumerate}
        \item $\mathscr G_\alpha^p\subseteq\mathscr G_\beta^p$;
        \item $\mathscr G_\alpha^p$ is closed under Rudin--Keisler images, that is, if $u\in\mathscr G_\alpha^p$ and $f:\omega\longrightarrow\omega$, then $f(u)\in\mathscr G_\alpha^p$.
    \end{enumerate}
\end{lemma}

\begin{proof}\hfill
    \begin{enumerate}
        \item The proof goes by induction on $\beta$, and the statement is only nontrivial if $\alpha<\beta$ and $\beta=\gamma+1$. By induction hypothesis we will have $\mathscr G_\alpha^p\subseteq\mathscr G_\gamma^p$, so we really only need to prove that $\mathscr G_\gamma^p\subseteq\mathscr G_{\gamma+1}^p$ for any ordinal $\gamma>0$. Now, given $u\in\mathscr G_\gamma^p$, define the constant sequence $v_n=u$ for all $n<\omega$. Then $u=\ulim{p}{n}u=\ulim{p}{n}v_n$, allowing us to conclude that $u\in\mathscr G_{\gamma+1}^p$.
        \item By induction on $\alpha$, with the case $\alpha=1$ being trivial (since Rudin--Keisler images of principal ultrafilters are again principal). The limit case is also straightforward, so we assume $\alpha=\gamma+1$ for a nonzero ordinal $\gamma$. If $u\in\mathscr G_{\gamma+1}^p$ then $u=\ulim{p}{n}u_n$ for some sequence of $u_n\in\mathscr G_\gamma^p$. Then, given an $f:\omega\longrightarrow\omega$, we have that
        \begin{equation*}
            f(u)=f(\ulim{p}{n}u_n)=\ulim{p}{n}f(u_n),
        \end{equation*}
        and by induction hypothesis each ultrafilter $f(u_n)\in\mathscr G_\gamma^p$. Therefore $f(u)\in\mathscr G_{\gamma+1}^p$.
    \end{enumerate}
\end{proof}

There is a couple of things we can conclude from Lemma~\ref{lem:galphaproperties}. The second point of the lemma implies that the family $\mathscr G^p$ is closed under Rudin--Keisler images; it is evident from the definition that it is also closed under taking ultrafilter limits along $p$. Also, the first point of Lemma~\ref{lem:galphaproperties} tells us that the stratification of $\mathscr G^p$ into $\mathscr G_\alpha^p$ provides a cumulative hierarchy for its elements, and so it makes sense to define the following rank-like function.

\begin{definition}
Given an ultrafilter $p\in\beta\omega$, for every $u\in\mathscr G^p$ we define
\begin{equation*}
    \rho_p(u)=\min\{\alpha\in\ord\big|u\in\mathscr G_{\alpha+1}^p\}.
\end{equation*}
\end{definition}

Alternatively, $\rho_p(u)$ is the unique ordinal number $\alpha$ such that $u\in\mathscr G_{\alpha+1}^p\setminus\mathscr G_\alpha^p$ (taking $\mathscr G_0^p=\varnothing$). Intuitively, the function $\rho_p$ (which we will call a rank function) counts the (possibly transfinite) number of times we need to take $p$-limits, starting from the principal ultrafilters, in order to obtain a certain element $u\in\mathscr G^p$. So the principal ultrafilters (and only them) have rank 0 (they are ``already there'' and we do not need to take $p$-limits to create them); all of the nonprincipal ultrafilters that are Rudin--Keisler below $p$ (including $p$ itself) have rank $1$, and so on. Note that, by the second part of Lemma~\ref{lem:galphaproperties}, $\eqrk{u}{v}$ implies $\rho_p(u)=\rho_p(v)$ for $u,v\in\mathscr G^p$.

Consider ultrafilters $p,u$ with $u\in\mathscr G^p$ nonprincipal, and let $\alpha=\rho_p(u)$. Then we have $u\in\mathscr G_{\alpha+1}^p$, so there exists a sequence of ultrafilters $u_n\in\mathscr G_\alpha^p$ such that $u=\ulim{p}{n}u_n$. This motivates the following definition.

\begin{definition}\label{def:gen-seq}
Given an ultrafilter $p\in\beta\omega$, and a $u\in\mathscr G^p$, we say that the sequence $\langle u_n\big|n<\omega\rangle$ of elements of $\mathscr G^p$ is a {\bf $p$-generating sequence} for $u$ if $u=\ulim{p}{n}u_n$ and $(\forall n<\omega)(\rho_p(u_n)<\rho_p(u))$.
\end{definition}

Note that every nonprincipal element of $\mathscr G^p$ necessarily admits a $p$-generating sequence, although this sequence need not be unique. In any case, it is not hard to check that $\rho_p(u)=\sup\{\rho_p(u_n)+1\big|n<\omega\}$ whenever $\langle u_n\big|n<\omega\rangle$ is a $p$-generating sequence for $u\in\mathscr G^p$. As a first application of Definition~\ref{def:gen-seq}, the following theorem strengthens the fact that $\mathscr G^p$ is closed under taking ultrafilter limits along $p$.

\begin{theorem}\label{thm:closedunderlimits}
    Let $p\in\beta\omega$ be an ultrafilter. Then, the family $\mathscr G^p$ is closed under ultrafilter limits. In other words, if $u\in\mathscr G^p$ and $u_n\in\mathscr G^p$ for all $n<\omega$, then $\ulim{u}{n}u_n\in\mathscr G^p$.
\end{theorem}

\begin{proof}
    If $u\in\mathscr G^p$, then we will proceed to prove the statement ``for every sequence of $u_n\in\mathscr G^p$, $\ulim{u}{n}u_n\in\mathscr G^p$'' by induction on $\rho_p(u)$. If $\rho_p(u)=0$ it means $u$ is a principal ultrafilter, say centred at $n_0$, and so $\ulim{u}{n}u_n=u_{n_0}$, hence the statement holds. Now if $\rho_p(u)\geq 1$, we let $\langle v_m\big|m<\omega\rangle$ be a generating sequence for $u$ and notice that 
    \begin{equation*}
    \ulim{u}{n}u_n=\ulim{\left(\ulim{p}{m}v_m\right)}{n} u_n=\ulim{p}{m}\left(\ulim{v_m}{n} u_n\right).
    \end{equation*}
    Since each $\rho_p(v_m)<\rho_p(u)$, we may apply the inductive hypothesis to conclude that, for each $n<\omega$, $w_m=\ulim{v_m}{n} u_n\in\mathscr G^p$, and consequently $\ulim{u}{n}u_n=\ulim{p}{n}w_n\in\mathscr G^p$.
\end{proof}

\begin{corollary}\label{cor:characterizationofg}
    For an ultrafilter $p\in\beta\omega$, $\mathscr G^p$ is the smallest family of ultrafilters on $\omega$ that
    \begin{enumerate}
        \item contains all principal ultrafilters,
        \item contains $p$,
        \item is closed under taking Rudin--Keisler images, and
        \item is closed under taking Blass--Frol\'{\i}k sums.
    \end{enumerate}
\end{corollary}

\begin{proof}
    Note that requirement (4) can be replaced, equivalently, with being closed under taking ultrafilter limits. From this perspective, it is clear that every element of $\mathscr G^p$ must belong to any family satisfying all four requirements, so it suffices to show that $\mathscr G^p$ itself satisfies them all. The first two are by definition (since all principal ultrafilters belong to $\mathscr G_1^p$ and $p$ belongs to $\mathscr G_2^p$), the third is point (2) of Lemma~\ref{lem:galphaproperties}, and the fourth, in its incarnation as being closed under taking ultrafilter limits, is Theorem~\ref{thm:closedunderlimits}.
\end{proof}

We will start deriving a few stronger facts under the hypothesis that the ultrafilter $p$ used to build our family $\mathscr G^p$ is selective. Notice, to begin with, that if $p$ is a selective ultrafilter then the generating sequence of any nonprincipal element of $\mathscr G^p$ is essentially unique, in the sense that any two generating sequences for that element will be $p$-equal by Theorem~\ref{thm:uniquenessofsequence}. So in this context, we will talk about ``the'' generating sequence of $u\in\mathscr G^p$. Note that this generating sequence must be $p$-injective: otherwise the sequence would be $p$-constant, but the $p$-constant sequence with constant value $v$ has $v$ itself as its $p$-limit, making it impossible to satisfy the requirement $\rho_p(v)<\rho_p(v)$ included in the definition of a $p$-generating sequence.

\begin{theorem}\label{thm:ifpselective}
    Let $p$ be a selective ultrafilter, let $u\in\mathscr G^p$ be nonprincipal, and let $\langle u_n\big|n<\omega\rangle$ be any sequence of ultrafilters satisfying $u=\ulim{p}{n}u_n$.  Then,
    \begin{enumerate}
        \item the sequence $\langle u_n\big|n<\omega\rangle$ is either $p$-constant, or $p$-equal to the $p$-generating sequence of $u$;
        \item there exists an $X\in p$ such that $\rho_p(u)\geq\sup\{\rho_p(u_n)\big|n\in X\}$ and, furthermore, if the sequence $\langle u_n\big|n<\omega\rangle$ is not $p$-constant, then in fact $\rho_p(u)=\sup\{\rho_p(u_n)+1\big|n\in X\}$.
    \end{enumerate}
\end{theorem}

\begin{proof}\hfill
\begin{enumerate}
    \item Since $p$ is selective, the sequence in question is either $p$-constant, or $p$-injective; in the latter case, it must be $p$-equal to any $p$-generating sequence of $u$ by Theorem~\ref{thm:uniquenessofsequence}.
    \item Since $p$ is selective, there is an $X\in p$ such that the values of the $u_n$ are either constant or one-to-one across $n\in X$. In the former case, we must actually have $u=u_n$ for all $n\in X$, and therefore the required inequality follows trivially (in fact, equality holds); in the latter case, by the previous point there is an $X\in p$ such that the sequence of $u_n$, for $n\in X$, is {\it the} $p$-generating sequence for $u$. So, redefining if necessary the terms of the generating sequence corresponding to $n\notin X$ (so that they now have e.g. $p$-rank zero), we obtain the equality claimed in the statement of the theorem.
\end{enumerate}
\end{proof}

We now consider the ultrafilter sum (in the sense of the algebra in the \v{C}ech--Stone compactification) in connection with the elements of $\mathscr G^p$. We begin with a small proposition that, although really easy to prove, is stated explicitly so as to not have to repeat the same computations multiple times.

\begin{proposition}\label{prop:talachasuma}
Let $p,u,v\in\beta\omega$ be ultrafilters, and let $\langle u_n\big|n<\omega\rangle$ be a sequence such that $u=\ulim{p}{n}u_n$. Then,
\begin{equation*}
    u+v=\ulim{p}{n}(u_n+v).
\end{equation*}
\end{proposition}

\begin{proof}
    \begin{eqnarray*}
    u+v & = & \ulim{u}{m}(v+m)=\ulim{\left(\ulim{p}{n}u_n\right)}{m}(v+m) \\
     & = & \ulim{p}{n}\left(\ulim{u_n}{m}(v+m)\right)=\ulim{p}{n}(u_n+v).
    \end{eqnarray*}
\end{proof}

Of course Proposition~\ref{prop:talachasuma} applies in the particular case where $u\in\mathscr G^p$ and $\langle u_n\big|n<\omega\rangle$ is a $p$-generating sequence for $u$. We will prove the strong result that, if the ultrafilter $p$ we started with is a selective ultrafilter, then the subsemigroup $\mathscr G^p$ is cancellative; along the way we will also establish that every element of $\mathscr G^p$ generates a free subsemigroup.

\begin{lemma}\label{lem:firstinequality}
    Let $p$ be a selective ultrafilter, and let $u,v\in\mathscr G^p$. Then $\rho_p(u+v)\geq\rho_p(v)$.
\end{lemma}

\begin{proof}
    By induction on $\rho_p(u)$. If $\rho_p(u)=0$ it means $u$ is a principal ultrafilter, so $\eqrk{u+v}{v}$ and in particular $\rho_p(u+v)=\rho_p(v)$. Now if $\rho_p(u)\geq 1$, let $\langle u_n\big|n<\omega\rangle$ be a $p$-generating sequence for $u$. By Proposition~\ref{prop:talachasuma}, we have $u+v=\ulim{p}{n}(u_n+v)$. By Theorem~\ref{thm:ifpselective}, there is an $X\in p$ such that $\rho_p(u+v)\geq\sup\{\rho_p(u_n+v)\big|n\in X\}\geq\rho_p(v)$, where the last inequality is justified by the inductive hypothesis.
\end{proof}

\begin{lemma}\label{lem:cancellative}
    Let $p$ be a selective ultrafilter. Then, the subsemigroup $\mathscr G^p$ of $\beta\omega$ is right cancellative, that is, whenever $u,v,w\in\mathscr G^p$ satisfy $u+w=v+w$, then $u=v$.
\end{lemma}

\begin{proof}
    The proof is done by induction on $\alpha=\max\{\rho_p(u),\rho_p(v)\}$. If $\alpha=0$ then both $u,v$ are principal ultrafilters, say centred at $n_0$ and $m_0$, respectively. Now suppose that $n_0+w=m_0+w$, let $N>\max\{n_0,m_0\}$ and let $i<N$ be such that $i+N\omega\in w$. Then $(n_0+i)+N\omega\in n_0+w$ and $(m_0+i)+N\omega\in m_0+w$, so these two sets must intersect and therefore $n_0+i\equiv m_0+i\mod N$, implying that $n_0\equiv m_0\mod N$; the choice of $N$ entails that we must have $u=n_0=m_0=v$.
    
    Now if $\alpha\geq 1$, the proof breaks into two cases, depending on whether both of $u,v$, or just one of them, are nonprincipal.
    
    \begin{description}
    \item[Case 1] Suppose that only one of $u,v$, without loss of generality $u$, is principal, centred at $n_0$. We must then have $v$ nonprincipal, so we may pick $\langle v_n\big|n<\omega\rangle$ a $p$-generating sequence for $v$. Then $u+w=n_0+w$, so $\eqrk{u+w}{w}$ and therefore $\rho_p(u+w)=\rho_p(w)$ by the second point of Lemma~\ref{lem:galphaproperties}. On the other hand, note that the sequence $\langle v_n+w\big|n<\omega\rangle$ must, by induction hypothesis, be $p$-injective and therefore we may use Proposition~\ref{prop:talachasuma} together with Theorem~\ref{thm:ifpselective} to conclude that this sequence is $p$-equal to a $p$-generating sequence for $v+w$. So, if $\langle w_n\big|n<\omega\rangle$ is a $p$-generating sequence for $v+w$, then we may assume that there exists an $X\in p$ such that $\rho_p(w_n)=0$ for $n\notin X$, and $w_n=v_n+w$ for $n\in X$; Lemma~\ref{lem:firstinequality} then tells us that $\rho_p(w_n)\geq\rho_p(w)$ and therefore $\rho_p(v+w)=\sup\{\rho_p(w_n)+1\big|n<\omega\}>\rho_p(w)=\rho_p(u+w)$. This makes it impossible to have $u+w=v+w$, and the proof of the case is done.
    \item[Case 2] Suppose now that $u,v$ are both nonprincipal ultrafilters, and pick generating sequences $\langle u_n\big|n<\omega\rangle$, $\langle v_n\big|n<\omega\rangle$ for each of them. Then, the induction hypothesis implies that $\langle u_n+w\big|n<\omega\rangle$ is a $p$-injective sequence, so that Proposition~\ref{prop:talachasuma} and Theorem~\ref{thm:ifpselective} allow us to conclude that it is a $p$-generating sequence for $u+w$. With the same argument one concludes that $\langle v_n+w\big|n<\omega\rangle$ is a $p$-generating sequence for $v+w$. Hence the assumption that $u+w=v+w$ implies that there is an $X\in p$ such that $u_n+w=v_n+w$ for all $n\in X$, so that once again by induction hypothesis we may conclude that $u_n=v_n$ for all $n\in X$. Hence the $p$-generating sequence for $u$ is $p$-equal to the $p$-generating sequence for $v$, which readily implies that $u=v$.
    \end{description}
\end{proof}

A part of the proof of Lemma~\ref{lem:cancellative} includes a reasoning that we will often use, so that we better state it once and for all to save space in the future.

\begin{corollary}\label{cor:generatingsum}
    Let $p$ be a selective ultrafilter, and let $u,v\in\mathscr G^p$. If $u$ is nonprincipal, with $p$-generating sequence $\langle u_n\big|n<\omega\rangle$, then $\langle u_n+v\big|n<\omega\rangle$ is $p$-equal to a $p$-generating sequence for $u+v$. In particular, there exists an $X\in p$ such that $\rho_p(u+v)=\sup\{\rho_p(u_n+v)+1\big|n\in X\}$.
\end{corollary}

\begin{proof}
     Lemma~\ref{lem:cancellative} implies that the sequence $\langle u_n+v\big|n<\omega\rangle$ is $p$-injective, and moreover its $p$-limit is precisely $u+v$ by Proposition~\ref{prop:talachasuma}. The conclusion follows directly from Theorem~\ref{thm:ifpselective}.
\end{proof}

The following lemma will be instrumental for the main result of this subsection.

\begin{lemma}\label{lem:rankofsums}
Let $p\in\beta\omega$ be a selective ultrafilter. If $u,v\in\mathscr G^p$, then $\rho_p(u+v)\geq\max\{\rho_p(u),\rho_p(v)\}$. Moreover the inequality is strict whenever $0<\rho_p(u)\leq\rho_p(v)$.
\end{lemma}

\begin{proof}
    By Lemma~\ref{lem:firstinequality}, it suffices to prove that $\rho_p(u+v)\geq\rho_p(u)$ and that, if $0<\rho_p(u)\leq\rho_p(v)$, then $\rho_p(u+v)>\rho_p(v)$. We prove these two statements simultaneously by induction on $\rho_p(u)$, with the base case $\rho(u)=0$ being obvious. In the case $\rho(u)>0$, we let $\langle u_n\big|n<\omega\rangle$ be a $p$-generating sequence for $u$, and use Corollary~\ref{cor:generatingsum} to find an $X\in p$ such that $\rho_p(u+v)=\sup\{\rho_p(u_n+v)+1\big|n\in X\}$. Now, $\rho_p(u_n+v)\geq\rho_p(u_n)$ for all $n\in X$ by induction hypothesis, so the conclusion is that $\rho_p(u+v)\geq\sup\{\rho_p(u_n)+1\big|n\in X\}=\rho_p(u)$.
    
    Now, if we further assume that $0<\rho_p(u)\leq\rho_p(v)$, the proof breaks in two cases, the first being if $\rho_p(u)=1$. In this case, each $u_n$ is principal, and hence $\eqrk{u_n+v}{v}$, so that $\rho_p(u_n+v)=\rho_p(v)$. Therefore $\rho_p(u+v)=\sup\{\rho_p(u_n+v)+1\big|n\in X\}=\rho_p(v)+1>\rho_p(v)$. Now, in the other case, where $\rho_p(u)>1$, there is at least one $n_0\in X$ such that $\rho_p(u_{n_0})>0$ and so by induction hypothesis $\rho_p(u_{n_0}+v)>\rho_p(v)$; therefore, we have $\rho_p(u+v)=\sup\{\rho_p(u_n+v)+1\big|n\in X\}>\rho_p(u_{n_0}+v)>\rho_p(v)$.
\end{proof}

The following easy consequence of Lemma~\ref{lem:rankofsums} is the main result of this subsection.

\begin{theorem}\label{thm:main-selective}
    Let $p\in\beta\omega$ be a selective ultrafilter. Then, every element of $\mathscr G^p$ generates a free subsemigroup. In particular, no element of $\mathscr G^p$ is an idempotent ultrafilter.
\end{theorem}

\begin{proof}
    If $u\in\mathscr G$ is nonprincipal and $n,m\in\mathbb N$ with $n<m$, then by Lemma~\ref{lem:rankofsums} we have $\rho(u^m)=\rho(u^n+u^{m-n})>\rho(u^n)$. In particular $u^m\neq u^n$.
\end{proof}

To finish the subsection, we state the following result that we believe is of independent interest.

\begin{theorem}\label{thm:cancellative}
    Let $p$ be a selective ultrafilter. Then, $\mathscr G^p$ is a cancellative subsemigroup. In other words, for any $u,v,w\in\mathscr G^p$ we have $u+w=v+w$ implies $u=v$, and $u+v=u+w$ implies $v=w$.
\end{theorem}

\begin{proof}
    Lemma~\ref{lem:cancellative} has already been proven, so it suffices to prove the second statement of the theorem; we will do so by induction on $\rho_p(u)$. The base case, $\rho(u)=0$ means $u$ is a principal ultrafilter, say centred at $n_0$. Hence we need to establish that, for $v,w\in\mathscr G^p$, $n_0+v=n_0+w$ implies $v=w$; this follows from~\cite[Lemma 6.28]{hindman-strauss}.
    
    To continue with the induction, we now assume that $\rho_p(u)\geq 1$ and pick a $p$-generating sequence $\langle u_n\big|n<\omega\rangle$ for $u$. Then, by Corollary~\ref{cor:generatingsum}, the sequence $\langle u_n+v\big|n<\omega\rangle$ is $p$-equal to the $p$-generating sequence for $u+v$ and the sequence $\langle u_n+w\big|n<\omega\rangle$ is $p$-equal to the $p$-generating sequence for $u+w$. Hence the assumption $u+v=u+w$ implies the existence of an $X\in p$ such that, for $n\in X$, $u_n+v=u_n+w$, in particular $X\neq\varnothing$ and the choice of any $n\in X$ together with the induction hypothesis yields $v=w$, and the proof is complete.
\end{proof}

Several of the arguments used in the proofs of Lemma~\ref{lem:cancellative} through Theorem~\ref{thm:cancellative} require an analysis of equations in $\beta\omega$ involving certain combinations of principal and nonprincipal ultrafilters as parameters. The interested reader may find a plethora of results in a similar vein in~\cite{maleki}. We finish this section by stating a question that might be of interest, in view of the proof of Lemma~\ref{lem:rankofsums}.

\begin{question}
Let $p$ be a selective ultrafilter, and let $u,v\in\mathscr G^p$. Is there a formula for $\rho_p(u+v)$ in terms of $\rho_p(u)$ and $\rho_p(v)$?
\end{question}

\subsection{A conjecture of Blass and an application to choiceless set theory}

Recall that a {\bf Solovay model} is one of the form $\lr$ as computed within the generic extension that results from L\'evy-collapsing some inaccessible cardinal to $\omega_1$. Mathias~\cite{mathias-happy} proved that, in such a model, a selective ultrafilter (outside of the model) is generic for $[\omega]^{\omega}/\mathrm{Fin}$, and so one might think of a model of the form $\lr[p]$, where $p$ is selective, as the result of having forced over $\lr$ (note that this does not add any new reals since the relevant forcing is $\sigma$-closed). Since one can perform, in $\zf$, the operations of taking Rudin--Keisler images---furthermore, in $\lr$ one can take {\it all} Rudin--Keisler images since functions $f:\omega\longrightarrow\omega$ are essentially reals--- and ultrafilter limits (equivalently, Blass--Frol\'{\i}k sums), it is clear that, if $p$ is a fixed selective ultrafilter, and $\mathscr G^p$ is as defined in the previous subsection, then one must have\footnote{Also note that, inductively, it is easy to show that $|\mathscr G_\alpha^p|=\mathfrak c$ for $\alpha>1$. Hence, in a model containing all reals it is also possible to consider all sequences of elements of $\mathscr G_\alpha^p$ in order to take their $p$-limits and obtain $\mathscr G_{\alpha+1}^p$.} $\mathscr G^p\subseteq\lr[p]$. Andreas Blass asked~\cite[Question, p. 251]{blass-selective} whether the ultrafilters belonging to $\lr[p]$ are precisely the elements of $\mathscr G^p$. Blass himself has conjectured\footnote{As told to the first author via personal communication.} for a long time that the answer to this question is affirmative. This would have various implications, among others, that in $\lr[p]$ the Rudin--Keisler ordering is linear, something that is impossible in models of $\zfc$ (see~\cite[Question, p. 251]{blass-selective}).

\begin{conjecture}[Blass]
If $\lr$ is a Solovay model, and $p$ is a selective ultrafilter on $\omega$, then the ultrafilters that belong to $\lr[p]$ are precisely the elements of $\mathscr G^p$. In other words, $\beta\omega^{\lr[p]}=\mathscr G^p$.
\end{conjecture}

So if this conjecture was true, Theorem~\ref{thm:main-selective} would provide some information about the existence of idempotent ultrafilters in models of the form $\lr[p]$.

\begin{theorem}
    If Blass's conjecture holds, then any model of the form $\lr[p]$, where $\lr$ is a Solovay model and $p$ is a selective ultrafilter on $\omega$, is a model of $\zf$ where there exist nonprincipal ultrafilters on $\omega$ but there are no idempotent ultrafilters.
\end{theorem}

The interest of this is that it answers a question~\cite[(5), p. 410]{dinasso-tachtsis} of Di Nasso and Tachtsis. We state the result below.

\begin{corollary}
    If Blass's conjecture holds, then the existence of nonprincipal ultrafilters does not imply the existence of idempotent ultrafilters.
\end{corollary}

As a final observation, we note that, if Blass's conjecture holds, then in models of the form $\lr[p]$ we have that $\beta\omega$ is a cancellative semigroup by Theorem~\ref{thm:cancellative}.

\section{Idempotent Ultrafilters on $\omega$ without ultrafilter lemma for $\mathbb{R}$}

Towards the end of the previous section, we pointed out that the results obtained have implications regarding the existence of idempotent ultrafilters without the Axiom of Choice. In this section we further pursue this vein of research, providing a proof of the consistency of $\zf$ together with $\neg\utr$ and the fact that every additive filter can be extended to an idempotent ultrafilter, thus answering~\cite[Question (2), p. 410]{dinasso-tachtsis} in the negative.

Recall~\cite[Def. 2.1]{dinasso-tachtsis} that a filter $\mathscr F\subseteq\wp(\omega)$ is said to be {\bf additive} if, for every ultrafilter $u\supseteq\mathscr F$, we have $\mathscr F\subseteq\mathscr F+u$, where the ``pseudo-sum''
\begin{equation*}
    \mathscr F+u=\left\{A\subseteq\omega\big| \{n<\omega\big|\{m<\omega\big|n+m\in A\}\in u\}\in\mathscr F\right\}
\end{equation*}
is defined in the exact same way as the usual sum, except applied to arbitrary filters rather than only to ultrafilters. It is an easy consequence of the Axiom of Choice (in fact, as proved in~\cite[Thm. 3.6]{dinasso-tachtsis}, this already follows from $\utr$) that every additive filter can be extended to an idempotent ultrafilter; in this section we will consider a symmetric model where every additive filter can be extended to an idempotent ultrafilter, but at the same time $\utr$ does not hold.

Therefore, we assume the reader has an appropriate background knowledge on symmetric submodels of generic extensions, as explained, e.g., in~\cite[Ch. 17]{halbeisen}. Any folklore lemmas, as well as any unexplained notations, will follow the aforementioned reference. 

The model we will use is a modification of the model without nonprincipal ultrafilters on $\omega$ from~\cite[pp. 391--392]{halbeisen}, having $\omega_1$ play the role of $\omega$. A general version of this model (using an arbitrary uncountable cardinal $\kappa$ in lieu of $\omega_1$) has been used elsewhere, e.g., in~ \cite[Thm. 3.1]{modelspectraofuniformity} and~\cite[pp. 12ss.]{usuba2024notelosstheorem}. We proceed to explain the model and, at the same time, fix notations that will remain unchanged throughout the rest of the section. Begin by assuming that $\mathbf{V}\vDash\ch$, and consider the forcing notion
\begin{equation*}
    \mathbb{P}=\left\{p;\omega_1 \times \omega_1\longrightarrow 2\big||\dom(p)|\leq \omega\right\}
\end{equation*}
partially ordered by reverse inclusion, i.e. $\leq$ simply means $\supseteq$. Note that this forcing notion is $\sigma$-closed and therefore does not add subsets of $\omega$. For every $X \subset \omega_1 \times \omega_1$, we define the automorphism $\sigma_X:\mathbb{P}\longrightarrow\mathbb P$ as follows:
\begin{equation*}
    \sigma_X(p)(x, y) =
                \begin{cases}
                    p(x, y) & \text{if } (x, y) \notin X, \\
                    1 - p(x, y) & \text{if } (x, y) \in X.
                \end{cases}
\end{equation*}
Note that, whenever $X,Y\subseteq\omega_1$, we have $\sigma_X \circ \sigma_Y = \sigma_{X \bigtriangleup Y}$. In fact, the mapping $X\longmapsto\sigma_X$ is a group isomorphism between $(\wp(\omega_1\times\omega_1),\bigtriangleup)$ and a subgroup of $\aut(\mathbb P)$, which we will denote by $\mathcal{G}=\left\{\sigma_X\big|X\subseteq \omega_1\times \omega_1\right\}$. Furthermore, for every $E\subseteq \omega_1$ we define $\fix(E)=\left\{\sigma_X\big|X\cap (E\times \omega_1)=\varnothing\right\}$. Note that the family $\left\{\fix(E)\big|E\subseteq \omega_1 \text{ is countable}\right\}$ generates a normal filter of subgroups of $\mathcal G$; we will denote this filter by $\mathscr F$; from here on out we will simply say that a $\mathbb P$-name is hereditarily symmetric without specifying that this is all with respect to $\mathcal G$ and $\mathscr F$. Our symmetric model $\mathbb M$ is obtained by letting $G$ be a $(\mathbf{V},\mathbb P)$-generic filter and setting
\begin{equation*}
    \mathbb M=\left\{\mathring{x}[G]\big|\mathring{x}\text{ is a hereditarily symmetric }\mathbb P\text{-name}\right\}.
\end{equation*}

\begin{theorem}[\cite{modelspectraofuniformity}, Theorem 3.1 case $\kappa=\omega_1$] \label{theorem_1}

$\mathbb{M}$ satisfies the $\ch$ plus there are no uniform ultrafilters on $\omega_1$.
\end{theorem}

It is worth noting that, even if~\cite[Thm. 3.1]{modelspectraofuniformity} includes $\gch$ as a hypothesis, its proof only uses $\gch$ up to the cardinal $\kappa$, which for us is $\omega_1$. Hence it suffices to assume $\mathbf{V}\vDash\ch$ to conclude that $\mathbb M\vDash\ch$ as well (although there is also no harm in outright assuming that $\mathbf{V}$ satisfies full $\gch$ throughout our proof). As for the fact that there are no uniform ultrafilters on $\omega_1$, the proof is essentially the same as in~\cite[pp. 391--392]{halbeisen} except it is ``stepped up'' from $\omega$ to $\omega_1$. The following corollary follows immediately.

\begin{corollary}
$\mathbb{M}$ satisfies that no ultrafilter on $\omega_1$ extends the filter of co-countable sets.
\end{corollary}

Finally, since $\mathbb R$ is equipotent to $\omega_1$ in $\mathbb M$, the following is also immediate.

\begin{corollary}\label{cor:noultrafiltertheorem}
$\mathbb{M} \models \neg\utr$.
\end{corollary}

So the remainder of the section is devoted to proving that, in $\mathbb M$, every additive filter on $\omega$ may be extended to an idempotent ultrafilter on $\omega$. To see this, it will be useful to decompose the forcing notion $\mathbb P$.

\begin{definition}
Let $\alpha<\omega_1$.
\begin{enumerate}
    \item We let $\mathbb P_\alpha$ be the {\it restriction} of the forcing notion $\mathbb P$ to ``the first $\alpha$ columns'', that is, 
    \begin{equation*}
    \mathbb P_\alpha=\left\{p\in \mathbb P\big|\dom(p)\subseteq\alpha\times\omega_1\right\}.
    \end{equation*}
    \item We similarly restrict the $(\mathbf{V},\mathbb P)$-generic filter $G$ to $\alpha$ by letting $G_\alpha=G\cap\mathbb P_\alpha$.
\end{enumerate}
\end{definition}

Note that, since conditions in $\mathbb P$ have countable domains, it follows that in fact $\mathbb P=\bigcup_{\alpha<\omega_1}\mathbb P_\alpha$. Standard arguments show that each $\mathbb P_\alpha$ is completely embedded in $\mathbb P$; therefore, $G_\alpha$ is a $(\mathbf{V},\mathbb P_\alpha)$-generic filter, and it makes sense to consider the forcing extension $\mathbf{V}[G_\alpha]=\left\{\mathring{x}[G_\alpha]\big|\mathring{x}\text{ is a }\mathbb P_\alpha\text{-name}\right\}$. Of course $\mathbf{V}[G_\alpha]\vDash\zfc$ and $\mathbf{V}[G_\alpha]\subseteq\mathbf{V}[G]$. We can prove even more.

\begin{lemma}\label{lem:intermediate-extensions}
For each $\alpha<\omega_1$, we have $\mathbf{V}[G_\alpha]\subseteq\mathbb M$.
\end{lemma}

\begin{proof}
It is immediate that $\mathring{x}[G_\alpha]=\mathring{x}[G]$ whenever $\mathring{x}$ is a $\mathbb P_\alpha$-name, so it suffices to show that every $\mathbb P_\alpha$-name is a hereditarily symmetric name. Note that, if $\sigma_X\in\fix(\alpha)$, then $\sigma_X(p)=p$ whenever $p\in\mathbb P_\alpha$. Therefore, it follows easily that $\fix(\alpha)\subseteq\sym_{\mathcal G}(\mathring{x})$ whenever $\mathring{x}$ is a $\mathbb P_\alpha$-name.
\end{proof}

We are now ready to tackle the main lemma of this section.

\begin{lemma}\label{lemma_j_1}
Let $x\in\mathbb M$ be such that $x\subseteq\mathbf{V}$. Then there exists an $\alpha<\omega_1$ such that $x\in\mathbf{V}[G_\alpha]$.
\end{lemma}

\begin{proof}
Let $\mathring{x}$ be a hereditarily symmetric $\mathbb P$-name for $x$. Define the following $\mathbb P$-name:
\begin{equation*}
    \mathring{y}=\left\{\left(\check{u},p\right)\big|p\Vdash``\check{u}\in\mathring{x}"\right\}.
\end{equation*}
It is straightforward to prove that $\mathring{y}[G]\subseteq\mathring{x}[G]=x$; conversely, if $u\in x=\mathring{x}[G]$ it is because $p\Vdash``\check{u}\in\mathring{x}"$ for some $p\in G$, therefore $(\check{u},p)\in\mathring{y}$ and so $u=\check{u}[G]\in\mathring{y}[G]$. So $\mathring{y}$ is a $\mathbb P$-name, such that $\mathring{y}[G]=x$, with the additional property that every element of $\dom(\mathring{y})$ is a check name. Furthermore, $\mathring{y}$ is a hereditarily symmetric name. To see this, suppose $E\subseteq\omega_1$ is a countable set such that $\fix(E)\subseteq\sym_{\mathcal G}(\mathring{x})$, and let $\sigma_X\in\fix(E)$. Then, for all $p\in\mathbb P$ satisfying $p\Vdash``\check{u}\in\mathring{x}"$, since $\sigma_X(\check{u})=u$ and $\sigma_X(\mathring{x})=\mathring{x}$, we conclude $\sigma_X(p)\Vdash``\check{u}\in\mathring{x}"$. Therefore, for each $(\check{u},p)\in\mathring{y}$, we have
\begin{equation*}
\sigma_X(\check{u},p)=\left(\sigma_X(\check{u}),\sigma_X(p)\right)=(\check{u},\sigma_X(p))\in\mathring{y};
\end{equation*}
so that $\sigma_X(\mathring{y})=\mathring{y}$. Thus $\fix(E)\subseteq\sym_{\mathcal G}(\mathring{y})$, as sought.

We now let $\alpha <\omega_1$ be such that $E\subseteq\alpha$. Define
\begin{equation*}
    \mathring{z}= \left\{(\check{u},p\upharpoonright\left(\alpha\times\omega_1)\right)\big|(\check{u},p)\in\mathring{y}\right\}
\end{equation*}
By definition, $\mathring{z}$ is a $\mathbb P_\alpha$-name and therefore $\mathring{z}[G]=\mathring{z}[G_\alpha] \in \mathbf{V}[G_\alpha]$. We claim that $\mathring{z}[G]=x$. It is straightforward to see that $x=\mathring{y}[G]\subseteq \mathring{z}[G]$, so we will focus on proving the reverse inclusion. For this, let $u=\check{u}[G]\in\mathring{z}[G]$, so that there is a $p\in\mathbb P$ with $(\check{u},p)\in\mathring{y}$, and $p\upharpoonright(\alpha\times\omega_1)\in G$. Take a $q\in G$ such that $q\leq p\upharpoonright(\alpha\times\omega_1)$ and $\dom(q)=\dom(p)$, and let $X=\left\{a\in \omega_1 \times \omega_1\big|p(a)\neq q(a)\right\}$. Observe that $X\cap(\alpha\times\omega_1)=\varnothing$ and that $\sigma_X (p)=q$. Therefore, $\sigma_X\in\fix(\alpha)\subseteq\fix(E)\subseteq\sym_{\mathcal G}(\mathring{y})$; from here it follows that  $(\check{u},q)=\sigma_X (\check{u},p)\in\mathring{y}$, and so $u=\check{u}[G]\in\mathring{y}[G]=x$. The proof is complete.
\end{proof}

From this, the following corollary follows relatively easily, and constitutes the result we have been seeking.

\begin{corollary}\label{cor:extendadditive}
$\mathbb{M}$ satisfies that every additive filter on $\omega$ can be extended to an idempotent ultrafilter. 
\end{corollary}

\begin{proof}
Let $\mathscr{F}$ be an additive filter on $\omega$, with $\mathscr F\in\mathbb{M}$. Of course we can think of $\mathscr F$ as a set of reals, and we know (since $\mathbb P$ does not add real numbers) $\mathbb M$ has the same reals as $\mathbf{V}$, so that $\mathscr F\subseteq\mathbf{V}$. Hence, by Lemma~\ref{lemma_j_1}, there exists an $\alpha <\omega_1$ such that $\mathscr F\in \mathbf{V}[G_\alpha]$. Since $\mathbf{V}[G_\alpha]$ satisfies the Axiom of Choice, there exists an idempotent ultrafilter $u\in\mathbf{V}[G_\alpha]$ that extends $\mathscr F$. By Lemma~\ref{lem:intermediate-extensions}, we have $u\in\mathbb M$; furthermore, we know that the reals in $\mathbb M$ are the same as the reals in $\mathbf{V}$ and so $u$ is still an ultrafilter in $\mathbb M$ (and of course $u+u$ is the same whether computed in $\mathbb M$ or in $\mathbf{V}$, so it is also an idempotent in $\mathbb M$), and we are done.
\end{proof}

We summarize the results from Corollaries~\ref{cor:noultrafiltertheorem} and~\ref{cor:extendadditive} in the following statement that answers~\cite[Question (2), p. 410]{dinasso-tachtsis}

\begin{corollary}
It is consistent with $\zf$ that every additive filter on $\omega$ can be extended to an idempotent ultrafilter yet $\neg\utr$ holds. In particular, the statement that every additive filter on $\omega$ can be extended to an idempotent ultrafilter does not imply $\utr$ under $\zf$.
\end{corollary}


\begin{thebibliography}{dinasso-tachtsis}


\bibitem{blass-thesis}
A. R. Blass,
{\em Orderings of ultrafilters.}
Ph.D. thesis, Harvard University, 1970.

\bibitem{blass-selective}
A. R. Blass,
{\em Selective ultrafilters and homogeneity.}
Ann. Pure Appl. Logic
{\bf 38} (1988), 215--255.

\bibitem{dinasso-tachtsis}
M. Di Nasso and E. Tachtsis, 
{\em Idempotent ultrafilters without Zorn's Lemma.}
Proc. Amer. Math. Soc.
{\bf 146} (2018), 397--411.

\bibitem{frolik}
Z. Frol{\'i}k,
{\em Sums of ultrafilters.}
Bull. Amer, Math, Soc.
{\bf 73} (1967), 87--91,

\bibitem{halbeisen}
L. J. Halbeisen,
{\em Combinatorial Set Theory. With a Gentle Introduction to Forcing, Second Edition.}
Springer Monographs in Mathematics, Springer, London, 2017.

\bibitem{modelspectraofuniformity}
Y. Hayut and  A. Karagila,
{\em Spectra of uniformity.}
Comment. Math. Univ. Carolin.
{\bf 60} No.~ 2 (2019), 285--298.

\bibitem{hindman-strauss}
N. Hindman and D. Strauss,
{\em Algebra in the {S}tone-\v {C}ech compactification. Second revised and extended edition}.
De Gruyter Textbook. Walter de Gruyter \& Co., Berlin, 2012.

\bibitem{maleki}
A. Maleki,
{\em Solving Equations in $\beta\mathbb N$.}
Semigroup Forum
{\bf 61} (2000), 373--384.

\bibitem{mathias-happy}
A. R. D. Mathias,
{\em Happy families.}
Annals of Mathematical Logic
{\bf 12} (1977), 59--111.

\bibitem{tachtsis-ellis}
E. Tachtsis,
{\em On the set-theoretic strength of Ellis' Theorem and the existence of free idempotent ultrafilters on $\omega$.}
The Journal of Symbolic Logic
{\bf 83}~No.~2 (2018), 551--571.

\bibitem{usuba2024notelosstheorem}
T. Usuba,
{\em A note on \L o\'s's Theorem without the Axiom of Choice.}
Bull. Pol. Acad. Sci. Math.
{\bf 72}~No.~1 (2024), 17--44.

\end{thebibliography}
\end{document}